\documentclass[11pt]{amsart}

\usepackage{graphicx}
\usepackage{epsfig}
\usepackage{amsmath}
\usepackage{amssymb}
\usepackage{bigints}
\usepackage{tikz}
\usepackage{cite}
\usepackage{mathtools}

\usetikzlibrary{arrows}

\theoremstyle{plain}
\newtheorem{theorem}                {Theorem}      [section]
\newtheorem*{theorem*}                {Theorem \ref{thm:appl}}
\newtheorem{proposition}  [theorem]  {Proposition}
\newtheorem{corollary}    [theorem]  {Corollary}
\newtheorem{lemma}        [theorem]  {Lemma}

\theoremstyle{definition}

\newtheorem{remark}       [theorem]  {Remark}
\newtheorem{definition}   [theorem]  {Definition}

\DeclareMathOperator{\trace}{trace}

\DeclareMathOperator{\Div}{div}

 \DeclareMathOperator{\Hess}{Hess}

\DeclareMathOperator{\grad}{grad}
\DeclareMathOperator{\Riem}{Riem}

\numberwithin{equation}{section}
\newcommand\myDelta{\stackrel{\mathclap{\tiny\mbox{$0$}}}{\Delta}}
\newcommand\mygrad{\stackrel{\mathclap{\tiny\mbox{$0$}}}{\grad}}
\newcommand\myRiem{\stackrel{\mathclap{\tiny\mbox{$N$}}}{\Riem}}

\begin{document}

\title[On biconservative surfaces]
{On biconservative surfaces}

\author{Simona~Nistor}

\address{Faculty of Mathematics - Research Department\\ Al. I. Cuza University of Iasi\\
Bd. Carol I, 11 \\ 700506 Iasi, Romania} \email{nistor.simona@ymail.com}

\thanks{The author was supported by a grant of the Romanian National Authority for Scientific Research and Innovation, CNCS - UEFISCDI, project number PN-II-RU-TE-2014-4-0004.}

\subjclass[2010]{Primary 53C42; Secondary 53C40}

\keywords{Biconservative surfaces, biharmonic submanifolds, mean curvature vector field, Codazzi tensor fields}

\begin{abstract}
We study in a uniform manner the properties of biconservative surfaces in arbitrary Riemannian manifolds. Biconservative surfaces being characterized by the vanishing of the divergence of a symmetric tensor field $S_2$ of type $(1,1)$, their properties will follow from general properties of a symmetric tensor field of type $(1,1)$ with free divergence. We find the link between the biconservativity, the property of the shape operator $A_H$ to be a Codazzi tensor field, the holomorphicity of a generalized Hopf function and the quality of the surface to have constant mean curvature. Then we determine the Simons type formula for biconservative surfaces and use it to study their geometry.

\end{abstract}

\maketitle
\section{Introduction}

In the last decade the theory of \textit{biconservative submanifolds} proved to be a very interesting research topic (see, for example, \cite{CMOP14,FNO16,FOP15,F15,FT16,MOR16,N16,S15,UT16}). This theory arose from the theory of \textit{biharmonic submanifolds}, but the class of biconservative submanifolds is richer than the later one.

Let $\left(M^m,g\right)$ and $\left(N^n,h\right)$ be two Riemannian manifolds. A \textit{biharmonic map} is a critical point of the \textit{bienergy functional}
$$
E_2:C^{\infty}(M,N)\rightarrow\mathbb{R},\quad E_{2}(\varphi)=\frac{1}{2}\int_{M}|\tau(\varphi)|^{2}\ v_g,
$$
where $\tau(\varphi)$ is the tension field of a smooth map $\varphi:M\to N$, and it is characterized by the vanishing of the \textit{bitension field} $\tau_2(\varphi)$ (see \cite{J86}). If $\varphi:\left(M^m,g\right)\to \left(N^n,h\right)$ is a biharmonic map and a Riemannian immersion, then $M$ is called a \textit{biharmonic submanifold} of $N$.

According to D. Hilbert (see \cite{H24}), to a functional $E$ we can associate a symmetric tensor field $S$ of type $(1,1)$, called the \textit{stress-energy tensor}, which is conservative, i.e., $\Div S=0$, at the critical points of $E$. In the particular case of the bienergy functional $E_2$, G. Y. Jiang (see \cite{J87}) defined the stress-bienergy tensor $S_2$ by
\begin{align*}
\langle S_2(X),Y\rangle=&\frac{1}{2}|\tau(\varphi)|^2\langle X,Y\rangle+\langle d\varphi,\nabla\tau(\varphi)\rangle\langle X,Y\rangle\\&-\langle d\varphi(X),\nabla_Y\tau(\varphi)\rangle-\langle d\varphi(Y),\nabla_X\tau(\varphi)\rangle,
\end{align*}
and proved that
$$
\Div S_2=-\langle\tau_2(\varphi),d\varphi\rangle.
$$
Therefore, if $\varphi$ is biharmonic, then $\Div S_2=0$ (see \cite{J87,LMO08}).

One can see that if $\varphi:\left(M^m,g\right)\to\left(N^n,h\right)$ is a Riemannian immersion then $\Div S_2=0$ if and only if the tangent part of the bitension field vanishes. A submanifold $M$ is called \textit{biconservative} if $\Div S_2=0$.

The biconservative submanifolds were studied for the first time in 1995 by Th. Hasanis and Th. Vlachos (see \cite{HV95}).

Biconservative submanifolds have some nice properties. For  example, when $m\neq 4$, a pseudoumbilical biconservative submanifold $\varphi:\left(M^m,g\right)\to \left(N^n,h\right)$ has constant mean curvature, i.e., it is $CMC$.  We have focused on the $m=2$ case and, in this special situation, the biconservative surfaces proved to have more interesting properties. A remarkable fact is that under the hypothesis of biconservativity some known results in the theory of submanifolds can be extended to more general contexts. For example, the generalized Hopf function associated to a $CMC$ biconservative surface in a Riemannian manifold is holomorphic (compare with the classical results: the Hopf function associated to a $CMC$ surface in a $3$-dimensional space form is holomorphic, and the generalized Hopf function associated to a $PMC$ surface in an $n$-dimensional space form is holomorphic).

The main idea is to notice that biconservative surfaces are characterized by $\Div S_2=0$ and therefore their properties will follow from the features of a free divergence symmetric tensor field of type $(1,1)$ on $M^2$ with a specific $\trace$.

The paper in organized as follows. After recalling some general results about tensor fields and submanifolds, we present in \textit{Section 3} some characterizations of biconservative submanifolds which satisfy some additional geometric hypotheses (as $A_H$ being a Condazzi tensor field or the surface having the mean curvature vector field $H$ parallel in the normal bundle, i.e., being $PMC$). We also study the properties of submanifolds with $A_H$ parallel, as they are automatically biconservative.

One of the main results in \textit{Section 4} is Theorem $\ref{teorema2}$ which gives a link between the biconservativity of a surface and some properties of $H$ and $A_H$. In order to obtain this result, as a biconsevative surface is characterized by the vanishing of the divergence of $S_2$, we study first, a little bit more generally, the properties of a symmetric tensor field $T$ of type $(1,1)$ with $\Div T=0$, and then we apply these properties to $S_2$. In this section we also pay a special attention to $CMC$ biconservative surfaces in an arbitrary manifold $N^n$. More precisely, we give a description, only in terms of $|H|$ and $\mu$ (where $\mu$ is the difference between the principal curvatures of $A_H$), of the metric on the surface and of the shape operator $A_H$. We prove that a $CMC$ biconservative surface can be immersed in $3$-dimensional space forms having as the shape operator the tensor field $A_H$ or $S_2$.

In \textit{Section 5} we find the expression of the rough-Laplacian $\Delta^R S_2$ of $S_2$ and then, integrating, we derive the conditions under which a compact biconservative surface has $A_H$ parallel (Theorem $\ref{main-th0.}$ and Theorem $\ref{main-th.}$). With a different technique we get a similar result in the complete non-compact case (Theorem $\ref{main-th1.}$).

\textit{Conventions.} Throughout this paper all manifolds, metrics and maps are assumed to be smooth, i.e., in the $C^\infty$ category, and we will often indicate the various Riemannian metrics by the same symbol $\langle,\rangle$. All manifolds are assumed to be connected and oriented. The following sign conventions for the curvature tensor field and for the rough-Laplacian are used
$$
R^N(U,V)T=\nabla^N_U\nabla^N_V T-\nabla^N_V\nabla^N_U T-\nabla^N_{[U,V]}T,
$$
and
$$
\Delta^\varphi W=-\trace \nabla^2 W,
$$
where $W\in C\left(\varphi^{-1}(TN)\right)$ and $U,V,T\in C(TN)$.

The mean curvature tensor field of a submanifold $M^m$ in $N^n$ is defined by
$$
H=\frac{1}{m}\trace B,
$$
where $B\in C(\odot^2 T^*M\otimes NM)$ is the second fundamental form of $M$ in $N$ and the $\trace$ is considered with respect to the metric on $M$.
\section{Preliminaries}\label{preliminaries}

First we recall some notions, formulas and general results about tensor fields and submanifolds that we will use later.

It is well-known that a symmetric tensor field $T$ of type $(1,1)$ on a Riemannian manifold $\left(M^m,g\right)$ can be identified with a symmetric tensor field $\tilde{T}$ of type $(0,2)$
$$
\langle T(X),Y\rangle=\tilde{T}(X,Y), \qquad X,Y\in C(TM),
$$
and, we will use the same notation $T$ instead of $\tilde{T}$.

\begin{proposition}
Let $\left(M^m,g\right)$ be a Riemannian manifold and consider $T$ and $S$ two symmetric tensor fields of type $(1,1)$. Then
\begin{equation}
\label{ec-prod-scalar-delta}
\langle\Delta^R T,S\rangle = \langle \nabla T,\nabla S \rangle - \Div Z,
\end{equation}
with $\Delta^R T = -\trace \nabla^2 T$, $Z\in C(TM)$, $Z=\langle \nabla_{X_i} T,S\rangle X_i$, where $\left\{X_i\right\}_{i=\overline{1,m}}$ is an orthonormal local frame field.
\end{proposition}

\begin{proposition}
Let $\left(M^m,g\right)$ be a Riemannian manifold and consider $T$ a symmetric tensor field of type $(1,1)$ and $\alpha$ is a smooth function on $M$. Then
\begin{equation}\label{diver}
\Div \left(T\left(\grad \alpha\right)\right) = \langle \Div T,\grad \alpha \rangle + \langle T, \Hess \alpha\rangle,
\end{equation}
\end{proposition}

\begin{definition}
A submanifold $\varphi:M^m\to N^n$ is called \textit{pseudoumbilical} if $A_H=|H|^2 I$, where $I$ is the identity tensor field of type $(1,1)$.
\end{definition}

Using the Codazzi equation, we easily find the next result.

\begin{proposition}
Let $\varphi:M^m\to N^n$ be a submanifold. Then
\begin{equation}
\label{subvarietateoarecare1}
\trace \nabla A_H = \frac{m}{2}\grad \left( |H|^2 \right) + \trace A_{\nabla_{\cdot}^\perp H}(\cdot)+ \trace\left(R^N(\cdot,H)\cdot\right)^T.
\end{equation}
\end{proposition}

\begin{corollary}
Let $\varphi:M^m\to N^n(c)$ be a submanifold, $c\in\mathbb{R}$. Then
\begin{equation}
\label{subvarietateoarecare2}
\trace \nabla A_H = \frac{m}{2}\grad \left( |H|^2 \right) + \trace A_{\nabla_{\cdot}^\perp H}(\cdot).
\end{equation}
\end{corollary}

Let $\varphi:M^m\to N^n$ be a submanifold. Computing $\tau_2(\varphi)$ by splitting it in the tangent and in the normal part and using \eqref{subvarietateoarecare1} we get the following characterizations for biconservative submanifolds (various expressions for $\tau_2(\varphi)$ were obtained in \cite{C84,LMO08,O02,O10}).

\begin{proposition}
\label{corolarbicons}
Let $\varphi:M^m\to N^n$ be a submanifold. Then the following conditions are equivalent:
\begin{enumerate}
    \item $M$ is biconservative;
    \item $\trace A_{\nabla^\perp_{\cdot} H}(\cdot)+\trace \nabla A_H +\trace \left( R^N(\cdot,H)\cdot\right)^T=0$;
    \item $\frac{m}{2}\grad\left(|H|^2\right)+2\trace A_{\nabla^\perp_{\cdot} H}(\cdot) + 2\trace \left( R^N(\cdot,H)\cdot\right)^T=0$;
    \item $2\trace \nabla A_H-\frac{m}{2}\grad\left(|H|^2\right)=0$.
\end{enumerate}
\end{proposition}

We end this section with the following result.

\begin{proposition}
Let $\varphi:M^m\to N^n$ be a submanifold. Then we have:

\begin{enumerate}
  \item the stress-bienergy tensor of $\varphi$ is determined by
   \begin{equation}
   \label{S_2}
    S_2= - \frac{m^2}{2}|H|^2 I + 2m A_H;
   \end{equation}
  \item $\trace S_2 = m^2|H|^2\left(2-\frac{m}{2}\right)$;
  \item the relation between the divergence of $S_2$ and the divergence of $A_H$ is given by
  \begin{equation}\label{divergente0}
  \Div S_2=-\frac{m^2}{2}\grad \left(|H|^2\right)+2m\Div A_H;
  \end{equation}
  \item $\left|S_2\right|^2=m^4|H|^4\left(\frac{m}{4}-2\right)+4m^2\left|A_H\right|^2$.
\end{enumerate}

\end{proposition}

\begin{remark}
From equation $\eqref{divergente0}$, we see that if $M$ is biconservative it does not follow that $\Div A_H$ automatically vanishes. In fact, only when $|H|$ is constant the biconservativity is equivalent to $\Div A_H=0$.

\end{remark}

\section{Other characterizations of biconservative submanifolds}

In this section we will characterize the biconservative submanifolds which satisfy some additional geometric hypotheses.

We begin with a study on the basic properties of submanifolds with $A_H$ parallel, as they are the ``simplest'' biconservative surfaces. First, we define the \textit{principal curvatures} of a submanifold $M^m$ of $N^n$ as being the eigenvalue functions of $A_H$.

\begin{proposition}
\label{propValProprii}
Let $\varphi:M^m\to N^n$ be a submanifold and $\lambda_1\geq \cdots \geq \lambda_m$ the principal curvatures of $M$. If $\nabla A_H=0$, then:
\begin{enumerate}
  \item $M$ is biconservative;
  \item $\lambda_i$ are constant functions on $M$, in particular $M$ is $CMC$;
  \item $A_{\nabla^\perp_X H}(Y)-A_{\nabla^\perp_Y H}(X)=\left(R^N(X,Y)H\right)^T$, for any $X,Y\in C(TM)$;
  \item $\trace A_{\nabla^\perp_\cdot H}(\cdot)=-\trace\left(R^N(\cdot,H)\cdot\right)^T$.
\end{enumerate}
\end{proposition}

\begin{proof}
For the sake of completeness, we give only the proof of the second item. More precisely, we show that $\lambda_i$ are constant functions on $M$. Let us consider an arbitrary point $p\in M$. Since $A_H(p)$ is symmetric, then $A_H(p)$ is diagonalizable. We denote by $\lambda_{1,p}\geq\cdots\geq\lambda_{m,p}$ the eigenvalues of $A_H(p)$ and then we define the continuous functions $\lambda_i:M\to \mathbb{R}$, $\lambda_i(p)=\lambda_{i,p}$, for any $p\in M$ and any $i=\overline{1,m}$.

Further, we consider $\left\{e_i\right\}_{i=\overline{1,m}}$ an orthonormal basis in $T_p M$ which diagonalize $A_H(p)$, i.e., $\left(A_H(p)\right)\left(e_i\right)=\lambda_i(p)e_i$, for any $i=\overline{1,m}$.

Consider $q\in M$, $q\neq p$, and $\gamma:[a,b]\to M$ a smooth curve such that $\gamma(a)=p$ and $\gamma(b)=q$. We define the vector fields $E_i=E_i(t)$, along $\gamma$, such that $DE_i(t)/dt=0$, for any $t$ and $E_i(a)=e_i$. It is easy to see that $W(t)=\left(A_H(\gamma(t))\right)\left(E_i(t)\right)$ is also a vector field along $\gamma$ and
\begin{align*}
 \frac{DW}{dt}(t)= & \left(\nabla_{\gamma'(t)}A_H\right)(E_i(t))+A_H\left(\frac{D E_i}{dt}(t)\right)\\
  = & \ 0.
\end{align*}
Now, since $D\left(\lambda_i(p)E_i\right)(t)/dt=0$, we get that $W(t)$ and $\lambda_i(p)E_i$ are parallel vector fields along $\gamma$. Since for $t=a$ they are equal, it follows that they coincide for any $t$, and in particular, for $t=b$. Therefore, $\lambda_i(p)$, $i=\overline{1,m}$, are eigenvalues of $A_H(q)$. As $q$ was chosen in an arbitrary way, we get that $\lambda_i$ are constant functions on $M$, for any $i=\overline{1,m}$.
\end{proof}

Later in this paper, we will find some converse results of this proposition, more precisely in the case when $m=2$, we will show that, under some ``standard'' hypotheses, a biconservative surface has $A_H$ parallel.

\begin{corollary}
Let $\varphi:M^m\to N^n(c)$ be a submanifold, $c\in\mathbb{R}$. If $\nabla A_H=0$, then
\begin{enumerate}
  \item $A_{\nabla^\perp_X H}(Y)=A_{\nabla^\perp_Y H}(X)$, for any $X,Y\in C(TM)$;
  \item $\trace A_{\nabla^\perp_\cdot H}(\cdot)=0$.
\end{enumerate}
\end{corollary}

In the particular case of surfaces, we get a stronger result.

\begin{proposition}
\label{flat-pseudo}
Let $\varphi:M^2\to N^n$ be a surface. If $\nabla A_H=0$, then $M$ is pseudoumbilical or flat.
\end{proposition}

\begin{proof}
Let $\lambda_1\geq\lambda_2$ be the principal curvatures of $M$. Since $\nabla A_H=0$, from Proposition $\ref{propValProprii}$ we have that $\lambda_1$ and $\lambda_2$ are constant functions on $M$.

If $\lambda_1=\lambda_2$, obviously, $M$ is pseudoumbilical.

If $\lambda_1>\lambda_2$, around any point of $M$ we can consider a local orthonormal frame field $\left\{E_i\right\}_{i\in\{1,2\}}$ which diagonalize $A_H$, i.e, $A_H\left(E_i\right)=\lambda_i E_i$, for $i\in\{1,2\}$.

Using Ricci's equation we get
$$
R(X,Y)A_H(Z)=A_H(R(X,Y)Z),
$$
for any $X,Y,Z\in C(TM)$ and then
\begin{align*}
  \lambda_i R\left(E_1,E_2\right)E_i= & A_H\left(R\left(E_1,E_2\right)E_i\right) \\
  = & A_H\left(R(E_1,E_2,E_1,E_i)E_1+R(E_1,E_2,E_2,E_i)E_2\right) \\
  = & \lambda_1 R(E_1,E_2,E_1,E_i)E_1+\lambda_2 R(E_1,E_2,E_2,E_i)E_2,
\end{align*}
for $i\in\{1,2\}$.

From the both choices of $i$, $i=1$ or $i=2$, we obtain $K=0$ on $U$, where $K$ is the Gaussian curvature of $M$ given by $K=R\left(E_1,E_2,E_1,E_2\right)$.
\end{proof}

If $\left(M^m,g\right)$ is a Riemannian manifold and $T$ is a parallel symmetric tensor field of type $(1,1)$, then its eigenvalue functions $\lambda_1\geq \cdots \geq \lambda_m$ are constant functions on $M$ and, obviously, $T$ is a Codazzi tensor field. If $m=2$, the converse also holds.

\begin{proposition}
Let $\left(M^2,g\right)$ be a surface and consider $T$ a symmetric tensor field of type $(1,1)$. Let $\lambda_1\geq \lambda_2$ be the eigenvalue functions of $T$. If $\lambda_1$ and $\lambda_2$ are constant functions on $M$ and $T$ is a Codazzi tensor field, then $\nabla T=0$. Moreover, if $\lambda_1>\lambda_2$, then $\left(M^2,g\right)$ is flat.
\end{proposition}

\begin{proof}
If $\lambda_1=\lambda_2=\lambda$, it follows that $T=\lambda I$ and $\nabla T=0$.

If $\lambda_1>\lambda_2$, around any point of $M$ we can consider a local orthonormal frame field $\left\{E_i\right\}_{i\in\{1,2\}}$ which diagonalize $T$. By a simple computation, one obtains
\begin{equation}
\label{ecValProp}
\left(\nabla_{E_j} T\right)\left(E_i\right) + T\left(\nabla_{E_j} E_i\right)=\lambda_i \nabla_{E_j}E_i,
\end{equation}
for $i,j\in \{1,2\}$.

Since $T$ is a Codazzi tensor field, for appropriate choices of $i$ and $j$ in $\eqref{ecValProp}$, we get $\nabla_{E_i}E_j=0$, for any $i,j\in \{1,2\}$. It follows that the Gaussian curvature of $M$ vanishes everywhere and $\nabla T=0$ on $M$.
\end{proof}

If $T=A_H$ we have the next result.

\begin{corollary}
\label{corolarAH}
Let $\varphi:M^2\to N^n$ be a surface and $\lambda_1\geq \lambda_2$ be  the principal curvatures of $M$. If $\lambda_1$ and $\lambda_2$ are constant functions on $M$ and $A_H$ is a Codazzi tensor field, then $\nabla A_H=0$.
\end{corollary}

\begin{remark}\label{const-CMC}
If for a submanifold $M^m$ in $N^n$ the principal curvatures of $M$ are constants, then $M$ is $CMC$.
\end{remark}

We note that, in general, $A_H$ is not a Codazzi tensor field. In the following we will study the properties of submanifolds with $A_H$ being a Codazzi tensor field, and their connection with biconservativity, as this is the next natural step after that of $A_H$ being parallel.

We begin with a result which follows easily from the Codazzi equation, equation $\eqref{subvarietateoarecare1}$ and Proposition $\ref{corolarbicons}$.

\begin{proposition}
\label{prop-calc1}
Let $\varphi:M^m\to N^n$ be a submanifold. If $A_H$ is a Codazzi tensor field then
\begin{enumerate}
  \item $A_{\nabla^\perp_X H}(Y)-A_{\nabla^\perp_Y H}(X)=\left(R^N(X,Y)H\right)^T$, for any $X,Y\in C(TM)$;
  \item $\trace \nabla A_H = m\grad \left(|H|^2\right)$;
  \item $\trace A_{\nabla^\perp_\cdot H}(\cdot)=\frac{m}{2}\grad\left(|H|^2\right)-\trace \left(R^N(\cdot,H)\cdot\right)^T$;
  \item $M$ is biconservative if and only $|H|$ is constant.
\end{enumerate}
\end{proposition}

If $N$ is an $n$-dimensional space form, Proposition $\ref{prop-calc1}$ can be rewritten as follows.

\begin{corollary}
Let $\varphi:M^m\to N^n(c)$ be a submanifold. If $A_H$ is a Codazzi tensor field then
\begin{enumerate}
  \item $A_{\nabla^\perp_X H}(Y)-A_{\nabla^\perp_Y H}(X) = 0$, for any $X,Y\in C(TM)$;
  \item $\trace \nabla A_H = m\grad \left(|H|^2\right)$;
  \item $\trace A_{\nabla^\perp_\cdot H}(\cdot)=\frac{m}{2}\grad\left(|H|^2\right)$;
  \item $M$ is biconservative if and only if $|H|$ is constant.
\end{enumerate}
\end{corollary}

We end this section considering those submanifolds in space forms having $H$ parallel in the normal bundle. Using the Codazzi equation, we get the following result.

\begin{proposition}
Let $\varphi:M^m\to N^n(c)$ be a submanifold with $\nabla^\perp H=0$. Then
\begin{enumerate}
  \item $M$ is a biconservative submanifold;
  \item $A_H$ is a Codazzi tensor field;
  \item $\langle \left(\nabla A_H\right)(\cdot,\cdot), \cdot\rangle$ is totally symmetric;
  \item $\trace \nabla A_H = 0$.
\end{enumerate}
\end{proposition}

\section{Properties of biconservative surfaces}

In this section we will study biconservative surfaces and determine some of their characteristic properties.

One of our main results is Theorem $\ref{teorema2}$. In order to prove it, we begin with a lemma which holds for an arbitrary symmetric tensor field $T$ of type $(1,1)$, we present some properties satisfied when $\Div T=0$, and then we finish by bringing them together in Theorem \ref{teorema1}.

\begin{lemma}
\label{lema1}
Let $\left(M^2,g\right)$ be a surface and let $T$ be a symmetric tensor field of type $(1,1)$. We have
\begin{enumerate}
  \item $$
        \Div T=\grad t-\langle Z_{12},X_2\rangle X_1+\langle Z_{12},X_1\rangle X_2,
        $$
      where $t=\trace T$, $\left\{X_1,X_2\right\}$ is a local orthonormal frame field on $M$ and
        $$
        Z_{12}=\left(\nabla_{X_1}T\right)\left(X_2\right)-\left(\nabla_{X_2}T\right)\left(X_1\right);
        $$
  \item $\langle T(\partial_z),\partial_z\rangle$ is holomorphic if and only if $\grad t=2\Div T$;
  \item $\langle T(\partial_z),\partial_z\rangle$ is holomorphic if and only if
        $$
        \grad t=2\langle Z_{12},X_2\rangle X_1-2\langle Z_{12},X_1\rangle X_2.
        $$
\end{enumerate}
\end{lemma}

\begin{proof}
Since the first and the third items follow by standard computation, we will only give the proof of the second item.

As $M$ is an oriented surface, locally, the metric $g$ can be written as $g=e^{2\rho}\left(dx^2+dy^2\right)$, where $(x,y)$ are local coordinates positively oriented and $\rho=\rho(x,y)$ is a smooth function. As usually, we denote
$$
\partial_z=\frac{1}{2}\left(\partial_x-i\partial_y\right) \text{ and } \partial_{\overline{z}}=\frac{1}{2}\left(\partial_x+i\partial_y\right).
$$
Therefore, $\langle T(\partial_z),\partial_z\rangle$ is holomorphic if and only if $\partial_{\overline{z}}\ \langle T(\partial_z),\partial_z\rangle=0$. We can see that the Christoffel symbols are given by
$$
\Gamma_{12}^2=\Gamma_{11}^1= -\Gamma_{22}^1 =\rho_x,
$$
where $\rho_x=\frac{\partial \rho}{\partial x}$ and
$$
\Gamma_{12}^1=\Gamma_{22}^2=-\Gamma_{11}^2=\rho_y.
$$
Thus, we obtain
\begin{align*}
  \nabla_{\partial_x}\partial_y = & \nabla_{\partial_y}\partial_x
   = \Gamma_{12}^1 \partial_x +\Gamma_{12}^2 \partial_y = \rho_y \partial_x +\rho_x\partial_y,\\\\
  \nabla_{\partial_x}\partial_x = & \Gamma_{11}^1 \partial_x +\Gamma_{11}^2 \partial_y
  = \rho_x \partial_x -\rho_y\partial_y, \\\\
  \nabla_{\partial_y}\partial_y = & \Gamma_{22}^1 \partial_x +\Gamma_{22}^2 \partial_y
  = -\rho_x \partial_x +\rho_y\partial_y.
\end{align*}
After some straightforward computations, we get
$$
\partial_{\overline{z}}\ \langle T(\partial_z),\partial_z\rangle =\frac{e^{2\rho}}{8}\left(-t_x + 2\langle \Div T,\partial_x \rangle +i\left(t_y - 2\langle \Div T,\partial_y \rangle \right)\right),
$$
where $t=\trace T$. Now, it is easy to see that $\langle T(\partial_z),\partial_z\rangle$ is holomorphic if and only if $\grad t=2\Div T$.

\end{proof}

\begin{remark}
We note that $\langle Z_{12},X_2\rangle X_1-\langle Z_{12},X_1\rangle X_2$ does not depend on the local orthonormal frame field $\left\{X_1,X_2\right\}$. Thus, there exists an unique global vector field $Z$ such that for any local orthonormal frame field, $\left\{X_1,X_2\right\}$, on its domain of definition we have
$$
Z=\langle Z_{12},X_2\rangle X_1-\langle Z_{12},X_1\rangle X_2.
$$
Therefore, we obtain the global formula
$$
\Div T=\grad t-Z.
$$
\end{remark}

Lemma $\ref{lema1}$ is the key ingredient to prove the following four propositions.

\begin{proposition}
\label{Prop-t=const}
Let $\left(M^2,g\right)$ be a surface and consider $T$ be a symmetric tensor field of type $(1,1)$. If $t$ is constant then the following relations are equivalent
\begin{enumerate}
  \item $T$ is a Codazzi tensor field;
  \item $\langle T\left(\partial_z\right),\partial_z \rangle$ is holomorphic;
  \item $\Div T=0$.
\end{enumerate}
\end{proposition}

\begin{proposition}
\label{Prop-divT=0}
Let $\left(M^2,g\right)$ be a surface and consider $T$ a symmetric tensor field of type $(1,1)$. If $\Div T=0$ then the following relations are equivalent
\begin{enumerate}
  \item $T$ is a Codazzi tensor field;
  \item $\langle T\left(\partial_z\right),\partial_z \rangle$ is holomorphic;
  \item $t$ is constant.
\end{enumerate}
\end{proposition}

\begin{proposition}
\label{Prop-funct-olomorfa}
Let $\left(M^2,g\right)$ be a surface and consider $T$ a symmetric tensor field of type $(1,1)$. If $\langle T\left(\partial_z\right),\partial_z \rangle$ is holomorphic then the following relations are equivalent
\begin{enumerate}
  \item $T$ is a Codazzi tensor field;
  \item $t$ is constant;
  \item $\Div T=0$.
\end{enumerate}
\end{proposition}

\begin{proposition}
\label{Prop-T-Codazzi}
Let $\left(M^2,g\right)$ be a surface and consider $T$ a symmetric tensor field of type $(1,1)$. If $T$ is a Codazzi tensor field then the following relations are equivalent
\begin{enumerate}
  \item $t$ is constant;
  \item $\langle T\left(\partial_z\right),\partial_z \rangle$ is holomorphic;
  \item $\Div T=0$.
\end{enumerate}
\end{proposition}

Summarizing, we can state the next theorem.

\begin{theorem}
\label{teorema1}
Let $\left(M^2,g\right)$ be a surface and consider $T$ be a symmetric tensor field of type $(1,1)$. Then any two of the following relations imply each of the others
\begin{enumerate}
  \item $\Div T=0$;
  \item $t$ is constant;
  \item $\langle T\left(\partial_z\right),\partial_z \rangle$ is holomorphic;
  \item $T$ is a Codazzi tensor field.
\end{enumerate}
\end{theorem}

Considering $T=S_2$ or $T=A_H$, we get the following result.

\begin{theorem}
\label{teorema2}
Let $\varphi:M^2\to N^n$ be a surface. Then any two of the following relations imply each of the others
\begin{enumerate}
  \item $M$ is biconservative;
  \item $|H|$ is constant;
  \item $\langle A_H\left(\partial_z\right),\partial_z \rangle$ is holomorphic;
  \item $A_H$ is a Codazzi tensor field.
\end{enumerate}
\end{theorem}

\begin{proof}
For the sake of completeness, we will point out a few details from the proof.

First, we recall that $S_2=-2|H|^2I+4A_H$, $\trace S_2=4|H|^2$, $\trace A_H=2|H|^2$ and
\begin{equation}\label{divergente}
\Div S_2=-2\grad \left(|H|^2\right)+4\Div A_H.
\end{equation}
It is easy to see that
$$
\langle S_2\left(\partial_z\right),\partial_z \rangle=4\langle A_H\left(\partial_z\right),\partial_z \rangle,
$$
and, therefore, $\langle A_H\left(\partial_z\right),\partial_z \rangle$ is holomorphic if and only if $\langle S_2\left(\partial_z\right),\partial_z \rangle$ is holomorphic.

The idea of the proof is to choose a condition and then prove the equivalence between each two other conditions using, in principal, Theorem $\ref{teorema1}$ applied for $T=A_H$ or $T=S_2$.

For example, we assume that $(1)$ holds. To prove that $(2)$ implies $(4)$ we note that since $\Div S_2=0$ and $|H|$ is constant, form $\eqref{divergente}$, we get $\Div A_H=0$. Now, from Theorem $\ref{teorema1}$ applied to $T=A_H$, we get $(3)$. Conversely, from Proposition $\ref{prop-calc1}$, we have that hypotheses $(1)$ and $(4)$ imply $(2)$. The other two equivalences, $(2)$ and $(3)$, and $(3)$ and $(4)$, respectively, follow easily from the equivalence between $\langle A_H\left(\partial_z\right),\partial_z \rangle$ being holomorphic and $\langle S_2\left(\partial_z\right),\partial_z \rangle$ being holomorphic, and the same Theorem $\ref{teorema1}$ with $T=A_H$ and $T=S_2$.

The other cases can be easily proved in a similar way.

\end{proof}

\begin{remark}
If $\varphi:M^2\to N^n$ is a non-pseudoumbilical $CMC$ biconservative surface then the set of pseudoumbilical points has no accumulation points. Also, we quickly deduce that if $M^2$ is a $CMC$ biconservative surface and it is a topological sphere then it is pseudoumbilical (see \cite{MOR16-2}); compare with the classical result: a $PMC$ surface $M^2$ of genius $0$ in a space form is pseudoumbilical (see \cite{H73}).
\end{remark}

Using Corollary $\ref{corolarAH}$, Remark $\ref{const-CMC}$ and Theorem $\ref{teorema2}$, one gets the next result.

\begin{theorem}
\label{main-th0.}
Let $\varphi:M^2\to N^n$ be a biconservative surface. We denote by by $\lambda_1$ and $\lambda_2$ the principal curvatures of $M$. If $\lambda_1$ and $\lambda_2$ are constant functions of $M$, then $\nabla A_H=0$.
\end{theorem}

\begin{remark}
If we replace the hypothesis of ``$M$ is a biconservative surface'' in Theorem $\ref{main-th0.}$ by ``$\langle A_H\left(\partial_z\right),\partial_z\rangle$ is a holomorphic function'' the conclusion still holds.
\end{remark}

Since, any $CMC$ surface in a $3$-dimensional space form is biconservative, using Theorem $\ref{teorema2}$, we easily get the following well-known properties for a $CMC$ surface.

\begin{corollary}
Let $\varphi:M^2\to N^3(c)$ be a $CMC$ surface, $c\in\mathbb{R}$. Then
\begin{enumerate}
  \item $M$ is biconservative;
  \item $A_H$ is a Codazzi tensor field;
  \item $\langle A_H\left(\partial_z\right),\partial_z\rangle$ is holomorphic.
\end{enumerate}
\end{corollary}

For a $PMC$ surface in an arbitrary manifold $A_H$ is not necessarily a Codazzi tensor field, but when the surface is biconservative, this does happen.

\begin{corollary}
Let $\varphi:M^2\to N^n$ be a biconservative surface with $\nabla^\perp H=0$. Then $A_H$ is a Codazzi tensor field and $\left(R^N(X,Y)H\right)^T=0$ for any $X,Y\in C(TM)$.
\end{corollary}

\begin{proof}
First, we note that $\nabla^\perp H=0$ implies $|H|$ constant, and, if $M$ is biconservative, from Theorem $\ref{teorema2}$, we have that $A_H$ is a Codazzi tensor field.

Now, to show that $\left(R^N(X,Y)H\right)^T=0$ we only have to replace $\nabla^\perp H=0$ in the Codazzi equation and use the fact that $A_H$ is a Codazzi tensor field.
\end{proof}

The next theorem gives a description in terms of $|H|$ and the difference between the principal curvatures of $M$ of the metric and of the shape operator in the direction of $H$ for a $CMC$ biconservative surface in an arbitrary manifold.

\begin{theorem}
\label{Delta-log-mu}
Let $\varphi:M^2\to N^n$ be a $CMC$ biconservative surface. Denote by $\lambda_1$ and $\lambda_2$ the principal curvatures of $M$ and by $\mu=\lambda_1-\lambda_2$ their difference. Then, around any non-pseudoumbilical point $p$ there exists a local chart $(U;x,y)$ which is both isothermal and a line of curvature coordinate system for $A_H$. Moreover, on $U$, we have
$$
\langle \cdot,\cdot\rangle=\frac{1}{\mu}\langle \cdot,\cdot\rangle_0,
$$
and $A_H$ is given by
$$
\langle A_H(\cdot),\cdot\rangle = \left(\frac{|H|^2}{\mu}+\frac{1}{2}\right) dx^2 + \left(\frac{|H|^2}{\mu}-\frac{1}{2}\right) dy^2,
$$
or, equivalently, by
$$
A_H = \left(\frac{|H|^2}{\mu}+\frac{1}{2}\right) dx \otimes \partial_x + \left(\frac{|H|^2}{\mu}-\frac{1}{2}\right) dy \otimes \partial_y,
$$
where $\langle\cdot,\cdot\rangle_0$ is the Euclidean metric on $\mathbb{R}^2$.

In the particular case when $n=3$, we obtain that $\mu$ satisfies
\begin{equation}\label{myDelta}
\mu \myDelta \mu + \left|\mygrad \mu\right|^2_0+2\mu\left(K^N+|H|^2-\frac{\mu^2}{4|H|^2}\right)=0,
\end{equation}
where $K^N$ is the sectional curvature of $N^3$ along $M^2$.
\end{theorem}

\begin{proof}
Let $\lambda_1$ and $\lambda_2$ be the principal curvatures of $M$ and $p$ a non-pseudoumbilical point in $M$. Since $\lambda_1$ and $\lambda_2$ are principal curvatures of $M$, they are continuous, and it follows that there exists an open neighborhood $U$ around $p$ such that $\lambda_1>\lambda_2$ are smooth functions on $U$ and $\mu=\lambda_1-\lambda_2$ is a positive smooth function on $U$.

Consider $\left\{E_1,E_2\right\}$ a local orthonormal frame field on $U$ such that $A_H\left(E_i\right)=\lambda_iE_i$, for any $i\in\left\{1,2\right\}$. Further, we consider the connection forms on $U$, defined by
$$
\nabla E_1=\omega_1^2 E_2 \text{ and } \nabla E_2=\omega_2^1 E_1.
$$
Clearly, $\omega_1^2=-\omega_2^1$. From Theorem $\ref{teorema2}$ one obtains $A_H$ is a Codazzi tensor field, i.e., on $U$ we have
$$
\left(\nabla A_H\right)\left(E_1,E_2\right)=\left(\nabla A_H\right)\left(E_2,E_1\right).
$$
Using the principal curvatures of $M$ and the definition of $\omega_i^j$, on $U$, we get
$$
\omega_1^2\left(E_1\right)=\frac{1}{\mu}E_2\left(\lambda_1\right) \text{ and } \omega_1^2\left(E_2\right)=\frac{1}{\mu}E_1\left(\lambda_2\right).
$$
Now, since $|\trace B|=2|H|$ is a constant, it is easy to see that $E_2\left(\lambda_1\right)=\left(E_2(\mu)\right)/2$ and $E_1\left(\lambda_2\right)=-\left(E_1(\mu)\right)/2$. If we denote by $\left\{\omega^1,\omega^2\right\}$ the local orthornormal coframe field defined on $U$ dual to $\left\{E_1,E_2\right\}$, one gets
$$
\omega^2_1=\frac{1}{2}\left(\left(E_2(\log \mu)\right)\omega^1 - \left(E_1(\log \mu)\right) \omega^2\right).
$$
After some straightforward computations, we obtain
$$
\left[ \frac{E_1}{\sqrt{\mu}}, \frac{E_2}{\sqrt{\mu}}\right]=0,
$$
and, therefore, on $U$ there exist coordinates functions $x$ and $y$ such that $\partial_x = E_1/\sqrt{\mu}$ and $\partial_y = E_2/\sqrt{\mu}$. Moreover the expression of the metric in isothermal coordinates on $U$ is
$$
\langle\cdot,\cdot\rangle =\frac{1}{\mu}\left(dx^2 + dy^2\right)=\frac{1}{\mu}\langle \cdot,\cdot\rangle_0.
$$
Since $\lambda_1$ and $\lambda_2$ are principal curvatures of $M$, it is easy to see that
$$
\langle A_H(\cdot),\cdot\rangle = \frac{1}{\mu}\left(\lambda_1 dx^2 +\lambda_2 dy^2\right).
$$
We conclude, using $\lambda_1 + \lambda_2 = 2|H|^2$ and $\lambda_1-\lambda_2=\mu$, that
\begin{equation}\label{val-proprii}
\lambda_1 = |H|^2+\frac{\mu}{2} \text{ and } \lambda_2 = |H|^2-\frac{\mu}{2}.
\end{equation}

In the $n=3$ case, from the Gauss equation it follows that
\begin{equation}\label{ec-gauss}
K = K^N\left(E_1,E_2\right)+|H|^2-\frac{\mu^2}{4|H|^2},
\end{equation}
where $K^N$ is the sectional curvature of $N$ along $M$.

Now, we recall that if $\langle \cdot,\cdot\rangle=e^{2\rho}\langle \cdot,\cdot\rangle_0$, where $\rho=\rho(x,y)$ is a smooth function on $M$, then $K=e^{-2\rho}\myDelta\rho$. In our case $\rho= -(\log \mu)/2$ and therefore
\begin{align*}
K & = -\frac{\mu}{2} \myDelta(\log \mu)\\
  & = -\frac{1}{2\mu}\left|\mygrad \mu\right|_0^2-\frac{1}{2}\myDelta\mu,
\end{align*}
where $\langle\cdot,\cdot\rangle_0$ is the Euclidean metric on $\mathbb{R}^2$, $\myDelta$ and $\mygrad$ are the Laplacian and the gradient, respectively, with respect to $\langle\cdot,\cdot\rangle_0$.

Therefore, replacing $K$ in $\eqref{ec-gauss}$, we obtain that $\mu$ is a solution of $\eqref{myDelta}$.
\end{proof}

\begin{remark}
Biconservative surfaces in Bianchi–-Cartan-–Vranceanu spaces, which are 3-dimensional spaces with non-constant sectional curvature, were studied in \cite{MOP}.
\end{remark}

\begin{remark}
If $N$ is a $3$-dimensional space form, the same result holds without imposing the hypothesis of biconservativity, as a $CMC$ surface is automatically biconservative.
\end{remark}

We note that, since $K=-(\Delta(\log \mu))/2$, the next result is obvious.

\begin{corollary}
\label{cololar-tor}
Let $\varphi:M^2\to N^n$ be a $CMC$ biconservative surface. Assume that $M$ is compact and does not have pseudoumbilical points. Then $M$ is a topologic torus.
\end{corollary}

Using Theorem $\ref{teorema2}$ and Corollary $\ref{cololar-tor}$ we obtain the following property.

\begin{corollary}
Let $\varphi:M^2\to N^n$ be a $CMC$ biconservative surface. Assume that $M$ is compact and does not have pseudoumbilical points. If $K\geq 0$ or $K\leq 0$, then $\nabla A_H=0$ and $K=0$.
\end{corollary}

We end this section with two results which basically say that a $CMC$ biconservative surface in $N^n$ can be also immersed in $N^3(c)$ having as the shape operator either the tensor field $A_H$ or $S_2$.

\begin{theorem}
Let $\varphi:M^2\to N^n$ be a biconservative surface. We denote by $\lambda_1$ and $\lambda_2$ the principal curvatures of $M$ corresponding to $\varphi$. Assume that $\lambda_1$ and $\lambda_2$ are constants and $\lambda_1>\lambda_2$. We have:
\begin{itemize}
  \item [a)] locally, there exists $\psi:M^2\to N^3(c)$ an isoparametric surface such that $A^\varphi_{H^\varphi}$ is the shape operator of  $\psi$ in the direction of the unit normal vector field, where
      $$
      c = \frac{\mu^2}{4}-\left|H^\varphi\right|^4;
      $$
      moreover $\left|H^\psi\right|=\left|H^\varphi\right|^2$.
  \item [b)] locally, there exists $\psi:M^2\to N^3(c)$ an isoparametric surface such that $S^\varphi_2$ is the shape operator of $\psi$ in the direction of the unit normal vector field, where
      $$
      c=4\left(\mu^2-\left|H^\varphi\right|^4\right);
      $$
      moreover $\left|H^\psi\right|=2\left|H^\varphi\right|^2$.
\end{itemize}
\end{theorem}

\begin{proof}
First, we define a symmetric tensor field $A^\psi$ of type $(1,1)$ on $M$ by
$$
A^\psi(X)=A^\varphi_{H^\varphi}(X), \qquad X\in C(TM).
$$
As $A^\varphi_{H^\varphi}$ is a Codazzi tensor field, $A^\psi$ satisfies, formally, the Codazzi equation for a surface in a $3$-dimensional space form.

Since the principal curvatures of $M$ corresponding to $\varphi$, $\lambda_1$ and $\lambda_2$, are constants and $\lambda_1>\lambda_2$, from Proposition $\ref{flat-pseudo}$ it follows that $K=0$. Now, formally, from the Gauss equation for a surface in a $3$-dimensional space form $N^3(c)$, and from $\eqref{val-proprii}$, we obtain
\begin{align*}
  c = & -\det A^\psi \\
  = & -\det A^\varphi_{H^\varphi}=\frac{\mu^2}{4}-\left|H^\varphi\right|^4.
\end{align*}
Therefore, locally, there exists an immersion $\psi:M^2\to N^3(c)$ such that its shape operator in the direction of the unit normal vector field is $A^\psi$. Moreover, the surface is isoparametric as $\lambda_1$ and $\lambda_2$ are constants.

It is known that $|\tau(\psi)|=2\left|H^\psi\right|=\trace A^\psi$ and in the same time
$$
\trace A^\varphi_{H^\varphi}=\lambda_1+\lambda_2=2\left|H^\varphi\right|^2=\frac{|\tau(\varphi)|^2}{2}.
$$
From the definition of $A^\psi$ we easily get $\left|H^\psi\right|=\left|H^\varphi\right|^2$.

Second, we define the shape operator associated to the surface $\psi:M^2\to N^3(c)$ as
$$
A^\psi(X)=S_2^\varphi(X), \qquad X\in C(TM),
$$
where $c\in\mathbb{R}$.

Since $S_2^\varphi=-2\left|H^\varphi\right|^2I+4A^\varphi_{H^\varphi}$, using the same argument as in the previous case, one obtains
\begin{align*}
  c = & -\det A^\psi \\
    = & -\det S^\varphi_2=4\left(\mu^2-\left|H^\varphi\right|^4\right)
\end{align*}
and  $|\tau(\psi)|=|\tau(\varphi)|^2$, i.e.,
$$
\left|H^\psi\right|=2\left|H^\varphi\right|^2.
$$
\end{proof}

\begin{theorem}
Let $\varphi:M^2\to N^n$ be a biconservative surface. Denote by $\lambda_1$ and $\lambda_2$ the principal curvatures of $M$ corresponding to $\varphi$. Assume that $\lambda_1$ and $\lambda_2$ are constants and $\lambda_1=\lambda_2$. If $K=0$, then we have:
\begin{itemize}
  \item [a)] Locally, there exists $\psi:M^2\to N^3(c)$ an umbilical surface such that $A^\varphi_{H^\varphi}$ is the shape operator of $\psi$ in the direction of the unit normal vector field, where
      $$
      c=-\left|H^\varphi\right|^4;
      $$
      moreover $\left|H^\psi\right|=\left|H^\varphi\right|^2$.
  \item [b)] locally, there exists $\psi:M^2\to N^3(c)$ an umbilical surface such that $S^\varphi_2$ is the shape operator of $\psi$ in the direction of the unit normal vector field, where
      $$
      c=-4\left|H^\varphi\right|^4;
      $$
      moreover $\left|H^\psi\right|=2\left|H^\varphi\right|^2$.
\end{itemize}
\end{theorem}

\section{The Simons type formula for $S_2$}

As we have already mentioned we will present here some converse results of Proposition $\ref{propValProprii}$ (see Theorem $\ref{main-th0.}$, Theorem $\ref{main-th.}$ for the compact case, and Theorem $\ref{main-th1.}$ for the complete non-compact case).

First, as in the previous section, we will compute the rough-Laplacian $\Delta^R T$ for an arbitrary symmetric tensor field $T$ of type $(1,1)$ on $M$ with $\Div T=0$, and then $\Delta^R S_2$.

\begin{proposition}
\label{prop-calc2}
Let $\left(M^2,g\right)$ be a surface and  $T$ a symmetric tensor field of type $(1,1)$. Assume that $\Div T=0$. Then
\begin{equation}
\label{ec-nabla^2}
\trace \left(\nabla^2 T \right) = 2 K T - t K I -(\Delta t) I- \nabla \grad t.
\end{equation}
\end{proposition}

\begin{proof}
Let $p\in M$ be an arbitrary point and $\left\{X_1,X_2\right\}$ a local orthornormal frame field, geodesic around $p$. Clearly, in $p$ we have
$$
\left(\trace\left(\nabla^2 T\right)\right)\left(X_j\right)=\sum_{i=1}^2\left(\nabla^2 T\right)\left(X_i,X_i,X_j\right).
$$
We note that we can rewrite the right hand term as
\begin{align*}
\sum_{i=1}^2\left(\nabla^2 T\right)\left(X_i,X_i,X_j\right)= &\sum_{i=1}^2\left(\left(\nabla^2 T\right)\left(X_i,X_i,X_j\right)-\left(\nabla^2 T\right)\left(X_i,X_j,X_i\right)\right)\\
   & + \sum_{i=1}^2 \left(\nabla^2 T\right)\left(X_i,X_j,X_i\right).
\end{align*}
After some straightforward computations, at $p$ one obtains
\begin{align*}
\sum_{i=1}^2\left(\nabla^2 T\right)\left(X_i,X_i,X_j\right)= &\sum_{i=1}^2\left(X_i\langle\Div T,X_j\rangle X_i- X_i\langle\Div T,X_i\rangle X_j-\left(X_i\left(X_jt\right)\right)X_i\right. +\\
 & \left.+ \left(X_i\left(X_it\right)\right)X_j + \left(\nabla^2 T\right)\left(X_i,X_j,X_i\right)\right),
\end{align*}
where $t=\trace T$, as in the previous section.

Further, applying Ricci's formula, since $\Div T=0$ and
$$
\sum_{i=1}^{2} \left( \nabla^2 T\right)\left(X_j,X_i,X_i\right)=0
$$
at $p$, it follows that, at $p$, we have
\begin{align*}
  \left(\trace \left(\nabla^2 T\right)\right)\left(X_j\right) =
  \left(2KT-\left(\Delta t\right) I-\nabla \grad t-KtI\right)\left(X_j\right).
\end{align*}
Therefore, at $p$ one obtains
$$
\trace \left(\nabla^2 T \right) = 2 K T - t K I -(\Delta t) I- \nabla \grad t.
$$
Since $p$ was arbitrary chosen, we get that the expression of $\trace \left(\nabla^2 T \right)$ holds on $M$.
\end{proof}

Using relation \eqref{ec-nabla^2} we can compute the Laplacian of the squared norm of $S_2$ and obtain a Simons type formula (here, instead of the second fundamental form we have the stress-bienergy tensor).

\begin{proposition}

Let $\varphi:M^2\to N^n$ be a biconservative surface. Then,
\begin{equation}\label{Simons}
\begin{array}{rl}
  \frac{1}{2}\Delta \left|S_2\right|^2 = & -2K\left|S_2\right|^2 + \Div \left( \left(\langle S_2,\grad \left(|\tau(\varphi)|^2\right)\rangle\right)^\sharp\right)+K|\tau(\varphi)|^4 \\
    & +\frac{1}{2}\Delta\left(|\tau(\varphi)|^4\right)+\left|\grad\left(|\tau(\varphi)|^2\right)\right|^2-\left|\nabla S_2\right|^2
\end{array}.
\end{equation}

\end{proposition}

\begin{proof}
First, using the fact that $\trace S_2=|\tau(\varphi)|^2$ and applying $(\ref{ec-nabla^2})$ for $T=S_2$ one obtains
\begin{equation}
\label{ec-DeltaS2}
\Delta^R S_2 = -2KS_2+\nabla\grad\left(|\tau(\varphi)|^2\right)+ \left(K|\tau(\varphi)|^2+\Delta \left(|\tau(\varphi)|^2\right) \right)I,
\end{equation}
where $\Delta^R S_2 = - \trace \left(\nabla^2 S_2 \right)$.

Then, from $\eqref{diver}$, since $M$ is biconservative, one gets
\begin{equation}\label{D1}
\Div \left( S_2\left(\grad \left(|\tau(\varphi)|^2\right)\right)\right)=
\langle S_2,\Hess \left(|\tau(\varphi)|^2 \right) \rangle.
\end{equation}
It is easy to see that from $(\ref{ec-prod-scalar-delta})$, considering $T=S=S_2$, one has
\begin{equation}\label{Weitz}
\frac{1}{2}\Delta \left|S_2\right|^2=\left\langle \Delta^R S_2,S_2\right\rangle-\left|\nabla S_2\right|^2.
\end{equation}
Further, since $\langle I,S_2\rangle = \trace S_2= |\tau(\varphi)|^2$ and
$$
\Delta\left(\left|\tau(\varphi)\right|^2\right)|\tau(\varphi)|^2= \frac{1}{2}\Delta\left(\left|\tau(\varphi)\right|^4\right)+\left|\grad\left(|\tau(\varphi)|^2\right)\right|^2,
$$
from $\eqref{ec-DeltaS2}$, $\eqref{D1}$ and $\eqref{Weitz}$, it follows that relation \eqref{Simons} holds.
\end{proof}

\begin{remark}
If $\varphi:M^2\to N^n$ is a $CMC$ biconservative surface, then $S_2$ is a Codazzi tensor field and relation \eqref{ec-DeltaS2} follows from a well-known formula in \cite{CY}.
\end{remark}

\begin{remark}
Formula $\eqref{ec-DeltaS2}$ was obtained in \cite{LO14} but for biharmonic maps (a stronger hypothesis) from surfaces and in a different way.
\end{remark}

Integrating \eqref{Simons} we get the following integral formula.

\begin{proposition}
Let $\varphi:M^2\to N^n$ be a biconservative surface and assume that $M$ is compact. Then
\begin{equation}
\label{int-S2}
\int_M \left( \left|\nabla S_2\right|^2 + 2K \left(\left|S_2\right|^2 - \frac{|\tau(\varphi)|^4}{2}\right) \right) v_g = \int_M \left|\grad \left(|\tau(\varphi)|^2\right)\right|^2 \ v_g.
\end{equation}
or, equivalently,
\begin{align*}
\int_M \left( \left|\nabla A_H\right|^2 + 2K\left(\left| A_H\right|^2-2|H|^4\right)\right)v_g=\frac{5}{2}\int_{M} \left|\grad \left(|H|^2\right)\right|^2 v_g.
\end{align*}
\end{proposition}

\begin{proof}
Since $M$ is compact, relation \eqref{int-S2} quickly follows integrating \eqref{Simons}. To obtain the second equation, i.e., an equivalent expression to $\eqref{int-S2}$, in terms of $A_H$ and $|H|$, we recall that
\begin{equation}\label{norm-s2}
\left|S_2\right|^2 = 16\left|A_H\right|^2 - 24|H|^4
\end{equation}
and
$$
\nabla_X S_2 = -2\left(X\left(|H|^2\right)\right) I + 4\nabla_X A_H,
$$
for any $X\in C(TM)$.

Then, by standard computations, we obtain
$$
\left|\nabla S_2\right|^2 = 16 \left|\nabla A_H\right|^2 -24 \left|\grad\left(|H|^2\right)\right|^2.
$$
Finally, we can rewrite $(\ref{int-S2})$ as
\begin{align*}
  \int_M \left( 16 \left|\nabla A_H\right|^2 -24 \left|\grad\left(|H|^2\right)\right|^2 + \right.& \left(2 K\left( 16\left|A_H\right|^2 - 24|H|^4 -8|H|^4 \right)\right) v_g  =\\
  = 16\int_M \left|\grad \left(|H|^2\right)\right|^2 v_g
\end{align*}
and by a direct computation we get
\begin{equation}
\label{int-AH}
\int_M \left( \left|\nabla A_H\right|^2 + 2K\left(\left| A_H\right|^2-2|H|^4\right)\right)v_g=\frac{5}{2}\int_{M} \left|\grad \left(|H|^2\right)\right|^2 v_g.
\end{equation}
\end{proof}

\begin{remark}\label{remark-poz}
It is easy to see that $2\left|S_2\right|^2-|\tau(\varphi)|^4=32\left(\left|A_H\right|^2-2|H|^4\right)$ is always non-negative, and it vanishes if and only if $S_2=\left(|\tau(\varphi)|^2\right/2)I$, or equivalently $M$ is pseudoumbilical.
\end{remark}

From $(\ref{int-AH})$ we easily get the following result.

\begin{theorem}
\label{main-th.}
Let $\varphi:M^2\to N^n$ be a $CMC$ biconservative surface and assume that $M$ is compact. If $K\geq 0$, then $\nabla A_H=0$ and $M$ is flat or pseudoumbilical.
\end{theorem}

\begin{proof}
Since $|H|$ is constant, from $(\ref{int-AH})$ one obtains
$$
\left|\nabla A_H\right|^2 + 2K\left(\left| A_H\right|^2-2|H|^4\right)=0.
$$
Therefore $\nabla A_H=0$ and $K\left(\left| A_H\right|^2-2|H|^4\right)=0$. From the last equality, it follows that $K=0$, i.e., $M$ is flat, or $\left| A_H\right|^2-2|H|^4=0$, i.e., $M$ is pseudoumbilical.
\end{proof}

In the following, we will study the complete non-compact biconservative surfaces.

\begin{theorem}
\label{main-th1.}
Let $\varphi:M^2\to N^n$ be a $CMC$ biconservative surface. Assume that $M$ is complete, non-compact, and $K\geq0$. If $\myRiem\ \leq k_0$, where $k_0$ is a non-negative constant, then $\nabla A_H=0$.
\end{theorem}

\begin{proof}

As $|H|$ is constant, $\tau(\varphi)=2H$ and $\left|S_2\right|^2=16\left|A_H\right|^2 - 24|H|^4$, from $\eqref{Simons}$ we get
\begin{equation}\label{Laplac-S2}
-\frac{1}{2}\Delta \left|S_2\right|^2=32K\left(\left|A_H\right|^2 - 2|H|^4\right)+\left|\nabla S_2\right|^2.
\end{equation}
Since $K$ and $\left|A_H\right|^2-2|H|^4$ are always non-negative (see the hypothesis and Remark $\ref{remark-poz}$, respectively), we get that $\Delta \left|S_2\right|^2\leq0$, i.e., $\left|S_2\right|^2$ is a subharmonic function.

Next, we prove that $\left|S_2\right|^2$ is bounded from above. From $\eqref{norm-s2}$ it is easy to see that $\left|S_2\right|^2$ is bounded from above if and only if  $\left|A_H\right|^2$ is bounded from above.

Let us consider $\left\{X_1,X_2\right\}$ a local orthonormal frame field on $M$ and $\left\{\eta,\eta_1,\cdots,\eta_{n-3}\right\}$ a local orthonormal coframe field on $M$ such that $H=|H|\eta$. From the Gauss equation we have
\begin{align*}
  K-\myRiem\left(X_1,X_2\right) = &  \langle B\left(X_1,X_1\right),B\left(X_2,X_2\right)\rangle - \left|B\left(X_1,X_2\right)\right|^2  \\
  = & \langle A_\eta\left(X_1\right),X_1\rangle \langle A_\eta\left(X_2\right),X_2\rangle-\left(\langle A_\eta\left(X_1\right),X_2\rangle\right)^2\\
   & +\sum_{\alpha=1}^{n-3}\left(\langle A_{\eta_\alpha}\left(X_1\right),X_1\rangle \langle A_{\eta_\alpha}\left(X_2\right),X_2\rangle-\left(\langle A_{\eta_\alpha}\left(X_1\right),X_2\rangle\right)^2\right)\\
  = & \det A_\eta+ \sum_{\alpha=1}^{n-3}\det A_{\eta_\alpha},
\end{align*}
where
$$
\myRiem\left(X_1,X_2\right) =R^N\left(X_1,X_2,X_1,X_2\right).
$$
It is clear that $\langle H,\eta_\alpha\rangle=0$, and then $\trace A_{\eta_\alpha}=2\langle H,\eta_\alpha\rangle=0$, for any $\alpha\in \left\{1,2,\cdots n-3\right\}$. As $A_{\eta_\alpha}$ is symmetric, we note that $\det A_{\eta_\alpha}\neq 0$ for any $\alpha$ and $\sum_{\alpha=1}^{n-3}\det A_{\eta_\alpha}\leq 0$. Then, we get
\begin{equation}\label{ineq-1}
K-\myRiem\left(X_1,X_2\right)\leq \det A_\eta.
\end{equation}
Let us consider $\mu_1$ and $\mu_2$ the principal curvatures of $A_\eta$. Then
\begin{align*}
  \det A_\eta =& \mu_1\mu_2=\frac{\left(\mu_1+\mu_2\right)^2-\left(\mu_1^2+\mu_2^2\right)}{2} \\
  = & \frac{\left(\trace A_\eta\right)^2-\left|A_\eta\right|^2}{2} \\
  = & \frac{4|H|^2-\left|A_\eta\right|^2}{2}.
\end{align*}
From $\eqref{ineq-1}$ one obtains
$$
\left|A_\eta\right|^2\leq 4|H|^2-2K+2\myRiem\left(X_1,X_2\right).
$$
Since $K\geq 0$ and $\myRiem\ \leq k_0$, it follows that
$$
\left|A_\eta\right|^2\leq 4|H|^2+2k_0.
$$
Therefore, $\left|A_\eta\right|^2$ is bounded from above by the constant $4|H|^2+2k_0$ and then $\left|A_H\right|^2$ is bounded from above.

It is well known that a complete surface with $K\geq 0$ is parabolic (see \cite{H57}), i.e., any subharmonic function bounded from above is constant. Thus, as $\left|S_2\right|^2$ is bounded from above and subharmonic, it follows that $\left|S_2\right|^2$ is a constant. Using
$\left|\nabla S_2\right|^2=16\left|\nabla A_H\right|^2$ and $\eqref{Laplac-S2}$ one obtains that $\nabla A_H=0$, and therefore $M$ is flat or pseudoumbilical.
\end{proof}

\subsection{Exemples of submanifolds with $\nabla A_H=0$}

As we have seen, a $PMC$ surface in a space form $N^n(c)$, $n\geq 4$, is trivially biconservative. But, if the surface is only $CMC$ then it is not necessarily biconservative. In \cite{MOR16-2} it was proved that if a surface is biconservative and $CMC$ in $N^4(c)$, with $c\neq0$, then the surface has to be $PMC$, i.e., the trivial case for our problem. We just recall here that, if $c=1$, then a $PMC$ surface in $\mathbb{S}^4$ is either a minimal surface of a small hypersphere of radius $a$, $a\in (0,1)$, in $\mathbb{S}^4$, or a $CMC$ surface in a small or great hypersphere in $\mathbb{S}^4$ (see \cite{Y74,Y75}). Of course, if we consider a $CMC$ biconservative surface $M^2$ of genus $0$ in $\mathbb{R}^4$, it is pseudoumbilical and therefore it is $PMC$, i.e., $M^2$ is a $2$-sphere (see \cite{H73}). In $\mathbb{R}^4$, there were obtained all $CMC$ biconservative surfaces which are not $PMC$. They are given by the isometric immersion $\varphi:\mathbb{R}^2\to\mathbb{R}^4$ defined by
$$
\varphi(u,v)=\overline{\gamma}(u)+(v+a)\overline{e}_4,
$$
where $\overline{\gamma}:\mathbb{R}\to\mathbb{R}^3$ is a smooth curve parametrized by arc length with positive constant curvature $k$ and free torsion $\tau$. By direct computation we obtain that the second fundamental form of the surface is given by
$$
B\left(\partial_u,\partial_u\right)=k(u)N(u), \qquad B\left(\partial_u,\partial_v\right)=0, \qquad B\left(\partial_v,\partial_v\right)=0,
$$
where $\{T(u),N(u),B(u)\}$ is the Frenet frame field associated to the curve $\overline{\gamma}$. Then, one obtains the expression of the mean curvature vector field
$$
H(u,v)=\frac{k}{2}N(u),
$$
the shape operator with respect to $H$
$$
A_H\left(\partial_u\right)=\frac{k^2}{2} \partial_u, \qquad A_H\left(\partial_v\right)=0
$$
and
$$
\nabla^\perp_{\partial_u}H=\frac{k}{2}\tau(u)N(u), \qquad \nabla^\perp_{\partial_v}H=0.
$$
It is easy to see that $\varphi$ is a biconservative immersion, i.e., satisfies
$$
\grad\left(|H|^2\right)+2\trace A_{\nabla_\cdot^\perp H}(\cdot)+2\trace \left(R^{\mathbb{R}^4}(\cdot,H)\cdot\right)^T=0.
$$
Therefore, $\varphi$ satisfies all hypotheses of Theorem $\ref{main-th1.}$ which implies that $A_H$ is parallel, a fact which can be also checked by a direct computation.

\end{document}